\documentclass{amsart}
\usepackage{amssymb}
\usepackage{amsmath}
\usepackage[matrix,arrow,curve,cmtip]{xy}

\theoremstyle{plain}   % This is the default, anyway

\newtheorem{theorem}{Theorem} % numbered theorem

\newtheorem{lemma}{Lemma}
\newtheorem{proposition}{Proposition}

\theoremstyle{definition}

\newtheorem{definition}{Definition}

\theoremstyle{remark}

\newtheorem{remark}{Remark}
\newtheorem{example}{Example}

\newcommand{\T}{\mathbb{T}}
\newcommand{\C}{\mathbb{C}}
\newcommand{\on}[1] {\operatorname{#1}}
\newcommand{\ind}[2] {\on{ind}\arrowvert_{#1}^{#2}}

\begin{document}
\title{Explicit formulas for 2-Characters}

\author{Ang\'{e}lica M. Osorno}

%\date{\today}
\begin{abstract}
Ganter and Kapranov associated a 2-character to 2-representations of a finite group. Elgueta classified 2-representations in the category of 2-vector spaces $2Vect_k$ in terms of cohomological data. We give an explicit formula for the 2-character in terms of this cohomological data and derive some consequences.
\end{abstract}

\maketitle

\section{Introduction}

In \cite{HKR}, Hopkins, Kuhn and Ravenel develop a theory of generalized characters that computes $E^*(BG)$ for the $n$-th Morava $E$-theory. The characters in this case are class functions defined on the set of $n$-tuples of commuting elements of $G$ whose order is a power of $p$. In \cite{nora}, Ganter and Kapranov define a 2-character for a 2-representation of a finite group in a 2-category. This 2-character is a function that assigns an element of the field $k$ to every pair $(g,h)$ of commuting elements in $G$. Ganter and Kapranov proved that these 2-characters satisfy the same formulas as the characters in \cite{HKR} for $n=2$.

The purpose of this paper is to find an explicit description of the 2-characters of a 2-representation in the 2-category of 2-vector spaces, $2Vect_k$. In order to find this description, we first review the algebraic classification of 2-representations. In \cite{elgueta}, it is shown that every equivalence class of 2-representations is given uniquely by a finite $G$-set $S$ and a class in $H^2(G;k^S)$. We present a streamlined approach to this result.

We then proceed to compute the 2-character in terms of this associated cohomology class. Using these computations we prove that 2-characters  are additive and multiplicative with respect to direct sum and tensor product of 2-representations. 

Given a 2-representation $\rho$ of $H\subseteq G$, Ganter and Kapranov also define the induced representation, and compute its character in terms of the character of $\rho$. Using our cohomological classification of representations, we identify the induced representation in terms of the cohomological data for $\rho$ using the Shapiro isomorphism.

Finally, using the explicit computation of 2-characters we give an example of two non-equivalent 2-representations that have the same character, thus showing that this assignment is not faithful.

The 2-categorical language can be cumbersome, so we review some of the terminology and constructions. For more background on 2-categories we refer the reader to \cite{nora} and \cite{elgueta}. 

The author would like to thank Nora Ganter for many helpful discussions that lead to some of the approaches and results presented here. The author would also like to thank Mark Behrens for his comments on earlier versions of this paper.

\section{2-representations and their characters}

Following \cite{nora}, we review the notions of 2-representations
of a group and character theory.

\begin{definition}\label{2rep} Let $\mathcal{C}$ be a 2-category and $G$ a group. A
\emph{2-representation} of $G$ in $\mathcal{C}$ is a lax 2-functor
from $G$ (viewed as a discrete 2-category) to $\mathcal{C}$.
\end{definition}

This amounts to the following data:

\begin{enumerate}
\item an object $V$ of $\mathcal{C}$,
\item for every $g\in G$, a 1-morphism $\rho _g:V\rightarrow V$,
\item a 2-isomorphism $\phi _1 : \rho _1 \Rightarrow \mathbf{1}
_{\mathcal{C}}$,
\item for every pair $g,h\in G$, a 2-isomorphism $\phi
_{g,h}:\rho _g \circ \rho _h \Rightarrow \rho_{gh}$.

This data has to satisfy the following conditions:
\item (associativity) for every $g,h,k \in G$,
$$\phi_{(gh,k)}(\phi_{g,h}\circ\rho_k) =
\phi_{(g,hk)}(\rho_g\circ\phi_{h,k}),$$

\item for any $g\in G$,
$$\phi_{1,g} = \phi_1\circ\rho_g \quad\text{and}\quad
        \phi_{g,1} = \rho_g\circ\phi_1.$$
\end{enumerate}

There is a 2-category in which we are particularly interested: the
2-category of 2-vector spaces. There are several 2-categories
which are 2-equivalent and give equivalent 2-representation
theories. We will use the definition in \cite{kapranov}.

\begin{definition}
Let $k$ be a field. The 2-category $2Vect_k$ has as objects $[n]$,
where $n\in \{0,1,2,\dots\}$. For integers $m,n$, the set of 1-morphisms
$1Hom_{2Vect_k}([m],[n])$ is the set of $m\times n$ matrices with
entries in $k$-vector spaces. These are called 2-matrices. Composition is given by matrix
multiplication using tensor product and direct sum. For 2-matrices
$A$ and $B$ of the same size, 2-morphisms are given by matrices of
linear maps $\phi$, with $\phi_{ij}:A_{ij}\rightarrow B_{ij}$.
\end{definition}

Thus a 2-representation of a group $G$ in $2Vect_k$ consists of the following data:

\begin{enumerate}
 \item A natural number $n$, called the dimension,
 \item for every $g\in G$, an $n\times n$ 2-matrix, $\rho _g$
 \item a 2-isomorphism $\phi _1 : \rho _1 \Rightarrow \mathbf{1}
_{[n]}$,
 \item for every pair $g,h\in G$, a 2-isomorphism $\phi
_{g,h}:\rho _g \circ \rho _h \Rightarrow \rho _{gh}$.
\end{enumerate}

The isomorphisms $\phi _{g,h}$ and $\phi _1$ are subject to the same conditions expressed above.

The following approach is similar to a more general result in \cite{elgueta}. The result is presented in a coordinate-free approach.

\begin{proposition}
 There is a one-to-one correspondence between equivalence classes of 2-representations of $G$ in $2Vect_{\C}$ and pairs $(S,[c])$ where $S$ is a finite $G$-set and $[c]\in H^2(G;(\C^{\times})^S)$. Here $(\C ^{\times})^S$ denotes $(\C ^{\times})^{|S|}$ as a $G$-module through the action of $G$ on $S$.
\end{proposition}

\begin{proof}
Note that since $\rho _g\rho _{g^{-1}}$ is isomorphic to $\mathbf{1}
_{[n]}$, each $\rho _{g}$ is given by a weakly invertible 2-matrix.
This means that the entries in $\rho _{g}$ can only be 0 and
1-dimensional vector spaces, with exactly one entry per row and
column being 1-dimensional. That is, up to isomorphism, $\rho _{g}$
is given by an $n\times n$ permutation matrix. Thus, we can
think of $\rho$ as a map

$$\rho:G\rightarrow \Sigma _n.$$

Now, let us turn our attention to $\phi _{g,h}$ and $\phi _1$. The 2-matrices $\rho _{gh}$ and $\rho _g\rho _h$ have only one
nonzero entry per row and column, and those entries are
1-dimensional vector spaces. Thus, to specify the 2-isomorphism
$\phi _{g,h}$ all we need to give is a sequence of $n$ nonzero
complex numbers $\{c_i(g,h)\}$ which give the isomorphism for the
nonzero entry in the $i$th row.

Condition (5) in the definition of a 2-representation implies

$$c_{\sigma^{-1}(i)}(h,k)\cdot c_i(g,hk)=c_i(gh,k)\cdot c_i(g,h),$$

where $\sigma$ is the permutation represented by $\rho _g$.

We can think of $(\C^{\times}) ^n$ as a $G$-module through $\rho$,
where $g\cdot \overrightarrow{a}=\rho _g\overrightarrow{a}$ in
matrix notation. We will denote this $G$-module by $(\C^{\times}
)_{\rho}^n$.

We can then think of $c$ as a 2-cochain $G\times G \rightarrow
(\C^{\times} )_{\rho}^n$. Then the condition above becomes the
cocycle condition:

$$(\delta c)(g,h,k)=g\cdot c(h,k)-c(gh,k)+c(g,hk)-c(g,h)=0$$

Here we are using additive notation for the component wise multiplication group structure of $\C^{\times} )^n$.

On the other hand, Condition (6) of Definition \ref{2rep} with $g=1$ implies that $\phi_1$
is given by multiplication by $c(1,1)$.

Hence, we can say that up to isomorphism, a 2-representation is
determined by a group homomorphism $\rho: G \rightarrow \Sigma _n$
and a 2-cocycle $c\in C^2(G;(\C^{\times} )_{\rho}^n)$. This
coincides with the notion in \cite{elgueta}.

In this new language, we would like to identify the equivalence
classes of representations. Two representations are equivalent if
there exists a pseudonatural equivalence between the vectors. A
pseudonatural transformation is a 1-morphism $f:[n]\rightarrow
[n']$ together with a 2-isomorphism $\psi(g):\rho '_g\circ f
\Rightarrow f\circ \rho _g$ for every $g\in G$, satisfying two
coherence conditions:

\begin{enumerate}
\item For all $g,h\in G$, $\psi(gh)\cdot (\phi_{g,h}'\circ \mathbf{1}_f)=(\mathbf{1}_f
\circ \phi _{g,h})\cdot (\psi(g)\circ \mathbf{1}_{\rho _h})\cdot
(\mathbf{1}_{\rho '_g}\circ \psi (h))$,
\item $\phi _1 '\circ \mathbf{1}_f=(\mathbf{1}_f\circ \phi
_1)\cdot \psi(1)$.
\end{enumerate}

This pseudonatural transformation is an equivalence if and only if
$f$ is a weakly invertible 1-morphism, that is if $n=n'$ and $f$
is given by a permutation matrix.

Assume two 2-representations are equivalent. If these 2-representations are given by the same map $\rho : G
\rightarrow \Sigma _n$ and $f=\mathbf{1}_{[n]}$, the 2-isomorphism
$\psi(g)$ is given by a sequence of $n$ nonzero complex numbers
$b_i(g)$ which give the isomorphisms on the nonzero 1-dimensional
vector spaces in each row. Again, we can think of these vectors of complex numbers
as a 1-cochain $G\rightarrow (\C^{\times}) ^n$. The two coherence
conditions imply for all $i$:

$$b_i(gh)c_i '(g,h)=c_i(g,h)b_i(g)b_{\sigma^{-1}(i)}(h),$$

where $c$ and $c'$ are the cocycles giving the two
representations. If we write this in additive notation we get:

$$(\delta b)(g,h)=g\cdot b(h)-b(gh)+b(g)=c'(g,h)-c(g,h).$$

That is, $c$ and $c'$ are cohomologous cocycles  in
$C^2(G;(\C^{\times} )_{\rho}^n)$ if and only if they give
equivalent representations.

In general the representations given by $\rho$, $[c]$ and $\rho '$, $[c']$ are equivalent if and only if
there exists a permutation $f\in \Sigma _n$ such that $\rho
'_g=f \rho_g f^{-1}$ and $[c']=[f\cdot c]$. This follows from the
assertions above. This proves the theorem.

\end{proof}

\section{Direct sum and tensor product}

There is a notion of direct sum and tensor product in $2Vect_k$ as noted in \cite{kapranov}. Direct sum is given as follows:

\begin{itemize}
 \item On objects: $[n]\oplus[m]=[n+m]$,
 \item on 1-morphisms is given by block sum of 2-matrices,
 \item on 2-morphisms is given by block sum of matrices of linear maps.
\end{itemize}

Tensor product is given as follows:

\begin{itemize}
 \item On objects: $[n]\otimes [n']=[nn']$.
 \item on 1-morphisms: let $f:[m]\rightarrow [n]$, $f':[m']\rightarrow [n']$ be 1-morphisms, then $f\otimes f':[mm']\rightarrow [nn']$ is the 2-matrix with $(i,i'),(j,j')$-entry equal to $f_{ij}\otimes f'_{i'j'}$, where the set of $mm'$ elements is labeled by pairs $(i,i')$, where $i=1,\dots, m$, $i'=1,\dots , m'$, with the order: $(1,1), (1,2), \dots , (1,n'), (2,1), \dots, (m,m')$ and similarly for $nn'$.
 \item on 2-morphisms: similarly as above.
\end{itemize}

These operations can be extended to 2-representations on $2Vect_k$ by taking the appropriate direct sum and/or tensor product of the respective objects, 1-morphisms and 2-morphisms. It is not hard to prove that we obtain a new 2-representation in both cases.

\section{Induced 2-Representations}

In \cite{nora}, Ganter and Kapranov define the notion of an induced representation given $H\subseteq G$ and inclusion of finite groups. Here we analyze the case of $2Vect_k$ following their explicit description in Remark 7.2.

Let $\rho : H\rightarrow \Sigma _n, [c]\in H^2(H; (\C ^{\times})^n _{\rho})$ be a 2-representation on $[n]$. Let $S$ be the corresponding $H$-set. Let $m$ be the index of $H$ in $G$. Let $\mathcal R = \{ r_1, \dots, r_m \}$ be a system of representatives of the left cosets of $H$ in $G$.

Then $\ind HG (\rho)$ is a 2-representation of $G$ of dimension $nm$. The matrix for $\ind HG \rho _g$ is a block matrix, with blocks of size $n \times n$, where the $(i,j)$-th block is given as follows:

$$(\ind HG\rho _g)_{ij}=
  \begin{cases}
    \rho _h & \text{if}\quad gr_j = r_ih, h\in H\\
    0       & \text{else.}
  \end{cases}
  $$.  
  
Now we turn our attention to $\ind HG \phi _{g_1,g_2}$. Notice that

$$
(\ind HG\rho _{g_1})\circ_1(\ind HG\rho _{g_2})_{ik} =
\phantom{XXXXXXXXXXXXXXXX}$$
$$ {\phantom{XXXXXX} = \begin{cases}
    \rho _{h_1}\circ_1\rho _{h_2} & \text{if}\quad g_1r_j = r_ih_1 \text{
      and }  g_2r_k = r_jh_2 \\
    0       & \text{else.}
  \end{cases}}
$$
  
and the $(i,k)$-th block is not zero precisely when since $g_1g_2r_k=r_ih_1h_2$, that is, when

$$(\ind HG \rho _{(g_1g_2)})_{ik}=\rho _{(h_1h_2)}.$$

Here $\circ_1$ denotes vertical composition of 2-morphisms.

On this block then 

$$(\ind HG \phi _{g_1,g_2})_{ik}=\phi _{h_1,h_2}.$$

Notice that the $G$-set given by $\ind HG \rho : G \rightarrow \Sigma _{nm}$ is precisely $\ind HG =G\times _H S\cong \mathcal R \times S$. Thus, the $G$-module $(\C^{\times})^{\ind HG S}$ is $\ind HG[(\C^{\times})^S]$.

The corresponding cocycle is then

$$(\ind HG c)_{(r_i,s)}(g_1,g_2)=c_s(h_1,h_2),$$

where $(r_i,s)\in \mathcal R \times S$ and $h_1, h_2$ are as above.

One can trace that $[\ind HG c]$ is the image of $[c]$ under the Shapiro isomorphism

$$H^2(H; (\C^{\times})^S) \cong  H^2 (G; \ind HG[(\C^{\times})^S]) \cong H^2 (G; (\C^{\times})^{\ind HG S}).$$

In particular, let $(S, [c])$ be an equivalence class of a representation of $G$. Let 

$$S=\coprod_{i=1}^k G/H_i$$

be the decomposition of $S$ into $G$ orbits. Then the chain of isomorphisms  

$$H^2(G; (\C^{\times})^S)\cong \bigoplus _{i=1}^k H^2(G; (\C^{\times})^{G/H_i}) \cong \bigoplus _{i=1}^k H^2(H_i; (\C^{\times}))$$

sends $[c]$ to $[c_1]\oplus \cdots \oplus [c_k]$ to $[d_1]\oplus \cdots \oplus [d_k]$, where $[d_i]$ is the image of $[c_i]$ under the Shapiro isomorphism.

This means that the representation given by $(G/H_i, [c_i])$ is the induced representation of a 1-dimensional representation $(\ast, [d_i])$ of $H_i$ 

We thus recover the following proposition, which coincides with \cite[Prop. 7.3]{nora}.

\begin{proposition}
 Every representation is the direct sum of induced 1-dimensional representations.
\end{proposition}

\section{2-characters}
We would like know what the character introduced in \cite{nora}
looks like in terms of $\rho$ and $c$.

\begin{definition}
Given a 2-category $\mathcal{C}$, let $A$ be an object and
$F:A\rightarrow A$ a 1-endomorphism. We define the
\emph{categorical trace} as

$$ \T r(F)=2Hom_{\mathcal{C}}(\mathbf{1} _A, F).$$

\end{definition}

Note that since $\mathcal{C}$ is a 2-category, $End(A)=Hom(A,A)$ is a category. This definition gives a functor $\T r : End(A)\rightarrow Set$.  If $\alpha : F \Rightarrow G$ is a 2-morphism between $F,G\in End(A)$, 

$$\T r(\alpha): 2Hom_{\mathcal{C}}(\mathbf{1} _A, F) \rightarrow 2Hom_{\mathcal{C}}(\mathbf{1} _A, G)$$

is given by composition with $\alpha$. 

Let $\mathcal{D}$ be a category. We recall that a 2-category is enriched over $\mathcal{D}$ if the categories $Hom(A,B)$ are enriched over $\mathcal{D}$ for all objects $A,B$.

If the 2-category $\mathcal{C}$ is enriched
over $\mathcal{D}$ then $\T r$ is a functor into
$\mathcal{D}$.

When $\mathcal{C}=Vect_k$, $A=[n]$ and $F$ is an $n\times n$
matrix $[F_{ij}]$ of vector spaces. Then the following equality holds

$$\T r(F)=\bigoplus _{i=1} ^n F_{ii}.$$

Note that the categorical trace is additive and multiplicative in the following sense. Let $F:[n]\rightarrow [n]$ and $G:[m]\rightarrow [m]$. Then $F\oplus G:[n+m]\rightarrow [n+m]$ and

$$\T r(F\oplus G)=\bigoplus _{i=1} ^{n+m} (F\oplus G)_{ii}=(\bigoplus _{i=1} ^n F_{ii}) \oplus (\bigoplus _{i=1} ^m G_{ii})=\T r(F) \oplus \T r(G).$$

Also, $F\otimes G:[n\cdot m]\rightarrow [n\cdot m]$ and

$$\T r(F\otimes G)=\negthickspace \negthickspace \negthickspace \bigoplus _{\substack{i=1,\dots,n\\ j=1,\dots,m}} \negthickspace \negthickspace \negthickspace (F\otimes G)_{(i,j)(i,j)}= \negthickspace \negthickspace \negthickspace \bigoplus _{\substack{i=1,\dots,n\\ j=1,\dots,m}} \negthickspace \negthickspace \negthickspace F_{ii} \otimes G_{jj}=(\bigoplus _{i=1} ^n F_{ii}) \otimes (\bigoplus _{i=1} ^m G_{jj})=\T r(F) \otimes \T r(F).$$

Given $F:A\rightarrow A$, let $G:A\rightarrow B$ be an equivalence
with quasi-inverse $H$. Then the trace is conjugation invariant
in the sense that there is an isomorphism:

$$\psi:\T r(F)\rightarrow \T r(GFH).$$

Let $\rho$ be a 2-representation of a group $G$ in $\mathcal{C}$. The character of
$\rho$ is given by assigning to each $g\in G$ the trace $\T r(\rho
_g)$. In this case, since we have the maps $\phi
_{g,h}$ and $\phi _1$ we can use the conjugation invariance above
to get a map

$$\psi _g (h): \T r(\rho _h)\rightarrow \T r(\rho _{ghg^{-1}}).$$

When $\rho$ is a 2-representation in $2Vect_k$ and $g$ and $h$
commute, $\T r(\rho _h)$ and $\T r(\rho _{ghg^{-1}})$ are the same
vector space. Thus Ganter and Kapranov introduce the following definition:

\begin{definition}
The \emph{2-character} of $\rho$ is a function on pairs
of commuting elements:

$$\chi _{\rho} (h,g)=trace\bigl(\psi _g (h): \T r(\rho _h)\rightarrow \T r(\rho
_h)\bigr).$$

\end{definition}

It is invariant under simultaneous conjugation.

Note that since $\T r$, $\psi _g (h)$, and $trace$ are all additive and multiplicative, we deduce that the character is also additive and multiplicative. (We will later give a different, explicit proof of this fact).

We are now prepared to compute the character of a 2-representation
 on $2Vect_{\C}$.

\begin{theorem}\label{big}
 The categorical character of a 2-representation given by $\rho :G\rightarrow \Sigma _n$ and $c\in
C^2(G;(\C^{\times} )_{\rho}^n)$ is

$$\chi (g,h)=\sum _{i=\rho _g(i)=\rho _h(i)} c_i(g,g^{-1})^{-1}c_i(1,1)^{-1}c_i(h,g^{-1})c_i(g,hg^{-1}).$$

\end{theorem}

\begin{proof}
Let $\varphi : \mathbf{1}_{[n]} \Rightarrow \rho _h$, that is, a
vector in $\T r(\rho _h)$. Note that this is an $n\times n$ matrix
with zero entries everywhere except in those diagonal entries that
are nonzero in $\rho _h$. Without loss of generality, we can
assume those are in the first $k$ rows (we can conjugate $\rho$ by
a permutation matrix $f$ and change $c$ accordingly to $f\cdot c$;
this will just give a reordering of the indices which does not
change the dimension of $\T r (\rho _h)$ nor the map $\psi _g(h)$). Thus we have

$$\varphi=
\begin{bmatrix}
  a_1 & \cdots & 0 & 0 & \cdots & 0\cr
  \vdots & \ddots & \vdots & \vdots & \ddots & \vdots\cr
  0 & \cdots & a_k & 0 & \cdots & 0\cr
    0 & \cdots & 0 & 0 & \cdots & 0\cr
  \vdots & \ddots & \vdots & \vdots & \ddots & \vdots\cr
  0 & \cdots & 0 & 0 & \cdots & 0\cr
 \end{bmatrix},$$

where $k$ is the number of indices fixed by $\rho_h$.

We would like to compute now

$$\psi_g(h)=\phi_{g,h,g^{-1}}\cdot (\rho_g\circ\varphi\circ\rho_{g^{-1}})\cdot\phi_{g,g^{-1}}^{-1}\cdot\phi_1^{-1}$$

for commuting pairs $(g,h)$. Note that since $g$ and $h$ commute, $\rho _h$ and $\rho_g\circ\rho_h\circ\rho_{g^{-1}}$ are isomorphic, in particular, the nonzero entries in the diagonal are in the same position. Thus we have

$$\rho _g\circ\varphi\circ\rho _{g^{-1}}=
\begin{bmatrix}
  a_{\rho _{g^{-1}}(1)} & \cdots & 0 & 0 & \cdots & 0\cr
  \vdots & \ddots & \vdots & \vdots & \ddots & \vdots\cr
  0 & \cdots & a_{\rho _{g^{-1}}(k)} & 0 & \cdots & 0\cr
    0 & \cdots & 0 & 0 & \cdots & 0\cr
  \vdots & \ddots & \vdots & \vdots & \ddots & \vdots\cr
  0 & \cdots & 0 & 0 & \cdots & 0\cr
 \end{bmatrix}.$$

On the other hand, composition with the isomorphisms $\phi$ is given just by multiplication by the appropriate scalar in the appropriate row:

$$\psi_g(h)=\phi_{g,h,g^{-1}}\cdot (\rho_g\circ\varphi\circ\rho_{g^{-1}})\cdot\phi_{g,g^{-1}}^{-1}\cdot\phi_1^{-1}=
\begin{bmatrix}
  b_1 & \cdots & 0 & 0 & \cdots & 0\cr
  \vdots & \ddots & \vdots & \vdots & \ddots & \vdots\cr
  0 & \cdots & b_k & 0 & \cdots & 0\cr
    0 & \cdots & 0 & 0 & \cdots & 0\cr
  \vdots & \ddots & \vdots & \vdots & \ddots & \vdots\cr
  0 & \cdots & 0 & 0 & \cdots & 0\cr
 \end{bmatrix},$$

where

$$b_i= c_i(g,g^{-1})^{-1}c_i(1,1)^{-1}c_i(h,g^{-1})c_i(g,hg^{-1})a_{\rho_g^{-1}(i)}.$$

We can think of the matrices $\{e_i\}_{i=1}^k$, where $e_i$ is the $n\times n$ matrix with 1 in the $(i,i)$ entry and zero everywhere else, as a basis for $\T r(\rho _h)$. Then the contribution to the character comes from the indices $i$ fixed by both $\rho _h$ and $\rho _g$:

$$\chi (g,h)=\sum _{i=\rho _g(i)=\rho _h(i)} c_i(g,g^{-1})^{-1}c_i(1,1)^{-1}c_i(h,g^{-1})c_i(g,hg^{-1}).$$

\end{proof}

\begin{lemma}
 The character is invariant under equivalence. 
\end{lemma}

\begin{proof}
Given two equivalent representations given by $\rho$, $[c]\in
H^2(G;(\C^{\times} )_{\rho}^n)$ and $\rho '$, $[c']\in
H^2(G;(\C^{\times} )_{\rho'}^n)$, there exists $f\in \Sigma _n$ such that $\rho
'=f^{-1} \rho f$ and $[c']=[f\cdot c]$, with $(\delta b)(g,h)=g\cdot b(h)-b(gh)+b(g)=c'(g,h)-f\cdot c(g,h).$

Then

\begin{align*}
\chi ' (g,h) & =\sum _{i=\rho _g '(i)=\rho _h '(i)}  c'_i(g,g^{-1})^{-1}c'_i(1,1)^{-1}c'_i(h,g^{-1})c'_i(g,hg^{-1})\\
& = \sum _{i=f\rho _gf^{-1}(i)=f\rho _hf^{-1}(i)}   \frac{b_i(1)}{c_{f^{-1}(i)}(g,g^{-1}) b_i(g^{-1})b_i(g)} \cdot \frac{b_i(1)}{c_{f^{-1}(i)}(1,1)b_i(1)b_i(1)}\\
& \qquad \qquad \qquad \cdot \frac{c_{f^{-1}(i)}(h,g^{-1}) b_i(g^{-1})b_i(h)}{b_i(hg^{-1})} \cdot \frac{c_{f^{-1}(i)}(g,hg^{-1})b_i(hg^{-1})b_i(g)}{b_i(h)}\\
& =\sum _{i=\rho _g(i)=\rho _h(i)}  c_{i}(g,g^{-1})^{-1}c_{i}(1,1)^{-1}c_{i}(h,g^{-1})c_{i}(g,hg^{-1})\\
&=\chi (g,h).
\end{align*}

\end{proof}

\begin{lemma}
The character respects the additive and multiplicative structure of representations.
\end{lemma}

\begin{proof}
Let $\rho, [c]$ and $\rho ', [c']$ represent two representations of dimensions $n,n'$.. The direct sum representation is given by $\tilde{\rho}_h$, the block sum of the matrices $\rho_h$ and $\rho '_h$ for every $h$ and the cocycle $\tilde{c}$ with

$$\tilde{c}_i=
\begin{cases}
 c_i & \text{if }   i\leq n,\\
 c'_{i-n} & \text{if }   i>n.
\end{cases}
$$

The character of the direct sum is

\begin{align*}
\tilde{\chi}(g,h)&=\sum _{i=\tilde{\rho} _g(i)=\tilde{\rho} _h(i)} \tilde{c}_i(g,g^{-1})^{-1}\tilde{c}_i(1,1)^{-1}\tilde{c}_i(h,g^{-1})\tilde{c}_i(g,hg^{-1})\\
&=\sum _{n\geq i=\rho _g(i)=\rho _h(i)} c_i(g,g^{-1})^{-1}c_i(1,1)^{-1}c_i(h,g^{-1})c_i(g,gh^{-1})\\
&+\sum _{\substack{i>n\\ i-n=\rho '_g(i-n)=\rho '_h(i-n)}} c'_{i-n}(g,g^{-1})^{-1}c'_{i-n}(1,1)^{-1}c'_{i-n}(h,g^{-1})c'_{i-n}(g,hg^{-1})\\
&=\chi(g,h)+\chi' (g,h).
\end{align*}

On the other hand, let $\overline{\rho},[\overline{c}]$ denote the tensor product of $\rho, [c]$ and $\rho ', [c']$. From the definition of the tensor product and using the labeling above for the set of $nn'$ elements, it is not hard to see that

$$(\overline{\rho}_g)_{(i,i'),(j,j')}=(\rho _g)_{i,j}(\rho '_g)_{i',j'},$$

$$\overline{c}_{(i.i')}=c_ic'_{i'}.$$

Then $(i,i')$ is fixed by $\overline{\rho}_h$ if and only if $i$ is fixed by $\rho _h$ and $i'$ is fixed by $\rho '_h$.

Thus

\begin{align*}
\overline{\chi}(g,h)&= \negthickspace \negthickspace \negthickspace \negthickspace \negthickspace
\sum _{(i,i')=\overline{\rho} _g(i,i')=\overline{\rho} _h(i,i')} \negthickspace \negthickspace \negthickspace \negthickspace \negthickspace \negthickspace \overline{c}_{(i,i')}(g,g^{-1})^{-1}\overline{c}_{(i,i')}(1,1)^{-1}\overline{c}_{(i, i')}(h,g^{-1})\overline{c}_{(i,i')}(g,hg^{-1})\\
&= \negthickspace \negthickspace \negthickspace \negthickspace \negthickspace \sum_{\substack{(i,i')\\ i=\rho _g(i)=\rho _h(i)\\ i'=\rho _g(i')=\rho _h(i')}} \negthickspace \negthickspace \negthickspace \negthickspace \negthickspace
((c_ic'_{i'})(g,g^{-1})(c_ic'_{i'})(1,1))^{-1} (c_ic'_{i'})(h,g^{-1})(c_ic'_{i'})(g,hg^{-1})\\
&=\biggl( \sum _{ i=\rho _g(i)=\rho _h(i)} c_i(g,g^{-1})^{-1}c_i(1,1)^{-1}c_i(h,g^{-1})c_i(g,hg^{-1}) \biggr)\\
&\quad \cdot \biggl( \sum _{ i'=\rho '_g(i')=\rho '_h(i')} c'_{i'}(g,g^{-1})^{-1}c'_{i'}(1,1)^{-1}c'_{i'}(h,g^{-1})c'_{i'}(g,hg^{-1})\biggr)\\
&=\chi(g,h)\chi' (g,h).
\end{align*}

Hence we see that the characters respect the additive an multiplicative structures on the representations.

\end{proof}

\begin{remark}
One can also use Theorem \ref{big} to reproduce the formula for the character of the induced representation which  appears in \cite[Corollary 7.6]{nora}:

 $$
    \chi_{\on{ind}}(g,h) = \frac{1}{|H|} \sum_{s^{-1}(g,h)s\in H\times
      H}\chi(s^{-1}gs,s^{-1}hs).
  $$
\end{remark}

One might hope that, analogous to the case of 1-characters,  the map from equivalence classes of 2-representations to characters is injective. This turns out not to be true, as the following counterexample shows.

\begin{example}
 
We will consider two 2-representations of $\Sigma _3$ of dimension 8 with trivial cocycle, so they amount to a group homomorphism $\rho : \Sigma _3 \rightarrow \Sigma _8$, that is, they are permutation representations. Note also that they are isomorphic are permutations representations if and only if they are isomorphic as 2-representations. Since they have trivial cocycle, the character is given by

$$\chi (g,h)=\sum _{i=\rho _g(i)=\rho _h(i)} 1=\#\{i=\rho _g(i)=\rho _h(i)\}.$$

Let $\rho $ be given by three blocks: the regular representation (action of $\Sigma _3$ on itself), and two trivial blocks. Let $\rho '$ be given by three blocks as well: 2 blocks with the action of $\Sigma _3$ on $\Sigma _3/\langle(12)\rangle$ and one block with the action of $\Sigma _3$ on $\Sigma _3/\langle (123) \rangle$.

Note that these two 2-representations are not isomorphic: $\Sigma _3$ fixes the last two elements of $\rho $ while $\Sigma _3$ fixes no element of $\rho '$.

%The tables of marks (see \cite{marks}) of these two representations is displayed below:

%\begin{center}
%\begin{tabular}{c|cccc}
%& $\langle 1\rangle$ & $\langle(12)\rangle$ & $\langle(123)\rangle$ & $\Sigma _3$\\
%\hline
%$\rho $ & 8 & 2 & 2 & 2\\
%$\rho '$ & 8 & 2 & 2 & 0
%\end{tabular}
%\end{center}

On the other hand, we can prove that these two representations have the same character: the pairs of commuting elements in $\Sigma _3$ are those containing 1 and $\{(123), (132)\}$.

We can directly compute the characters: $\chi  (1,1)=\chi ' (1,1)=8$; $\chi (1,g)=\chi '(1,g)=2$ for all $g\neq 1$ and $\chi ((123),(132))=\chi '((123),(132))=2$.
\end{example}

\end{document}